\documentclass[12pt]{amsart}
\usepackage{times}
\usepackage{amsfonts, amsmath, amssymb, amsthm, mathtools, dsfont}
\usepackage[colorlinks=true,citecolor=blue]{hyperref}
\usepackage{setspace}
\usepackage{enumerate}
\usepackage[utopia]{mathdesign}
\usepackage[mathscr]{euscript}
\usepackage[utf8]{inputenc}
\usepackage{tikz}
\usetikzlibrary{arrows}
\usetikzlibrary{decorations.pathreplacing,decorations.markings}

\tikzset{
  on each segment/.style={
    decorate,
    decoration={
      show path construction,
      moveto code={},
      lineto code={
        \path [#1]
        (\tikzinputsegmentfirst) -- (\tikzinputsegmentlast);
      },
      curveto code={
        \path [#1] (\tikzinputsegmentfirst)
        .. controls
        (\tikzinputsegmentsupporta) and (\tikzinputsegmentsupportb)
        ..
        (\tikzinputsegmentlast);
      },
      closepath code={
        \path [#1]
        (\tikzinputsegmentfirst) -- (\tikzinputsegmentlast);
      },
    },
  },
  mid arrow/.style={postaction={decorate,decoration={
        markings,
        mark=at position .5 with {\arrow[#1]{stealth}}
      }}},
}

\usepackage{graphicx}
\usepackage{parskip}
\usepackage{enumerate}
\usepackage{verbatim}
\usepackage{units}

\usetikzlibrary{calc}
\usepackage[outline]{contour}
\contourlength{1.2pt}

\begingroup
    \makeatletter
    \@for\theoremstyle:=definition,remark,plain\do{%
        \expandafter\g@addto@macro\csname th@\theoremstyle\endcsname{%
            \addtolength\thm@preskip\parskip
            }%
        }
\endgroup

\newtheorem{theorem}{Theorem}[section]
\newtheorem*{theorem*}{Theorem}
\newtheorem{lemma}[theorem]{Lemma}

\newtheorem{cor}[theorem]{Corollary}

\theoremstyle{definition}

\newtheorem*{remark*}{Remark}
\newtheorem{claim}[theorem]{Claim}

\newcommand{\mes}{\text{mes}}

\title{Collapsibility of noncover complexes of chordal graphs}

\author{Jinha~Kim}
\address{J.~Kim \\ Department of Mathematical Sciences \\ Seoul National University \\ Seoul \\ Republic of Korea} \email{kjh1210@snu.ac.kr}

\date\today

\begin{document}
\begin{abstract}
Let $G$ be a graph on $V$.
A vertex subset $S \subset V$ is called a cover of $G$ if its complement is an independent set, and $S$ is called a noncover if it is not a cover of $G$.
A noncover complex $NC(G)$ of $G$ is the simplicial complex on $V$ whose faces are noncovers of $G$.
The independence domination number $ i\gamma(G)$ of $G$ is the minimum integer $k$ such that every independent set of $G$ can be dominated by $k$ vertices.
In this note, we prove that $NC(G)$ is $(|V|- i\gamma(G)-1)$-collapsible.
\end{abstract}

\maketitle

\section{Introduction}
For a graph $G=(V,E)$, a vertex subset $I \subset V$ is called an {\it independent set} if there is no adjacent pair of vertices of $I$.
A {\em (vertex) cover} of $G$ is a vertex subset $W \subset V$  such that $V \setminus W$ is independent in $G$.
Note that any vertex cover of $G$ meets every edge of $G$.
$W \subset V$ is called a {\em noncover} of $G$ if it is not a cover of $G$.
A {\em noncover complex} of $G$ is a simplicial complex defined as follows:
\[NC(G) := \{W \subset V(G) : W \text{ is a noncover in } G\}.\]
In this note, we prove a bound for the ``collapsibility number'' of noncover complexes.

\subsection{Independence domination numbers}
For a vertex subset $W \subset V$ in a graph $G$, let $N(W)$ be the set of neighbors of vertices in $W$, i.e.
\[N(W) := \{v \in V: v\text{ is adjacent to }w\text{ for some }w\in W\}.\]
We say $W$ {\em dominates} a vertex $v$ in $G$ if $v \in W \cup N(W)$, and we say $W$ dominates $A \subset V$ in $G$ if $W$ dominates every vertex in $A$, i.e. $A \subset N(W)$.
We say $W$ dominates $G$, if $W$ dominates $V(G)$ in $G$.
The {\em domination number} $\gamma(G)$ is defined by
\[
\gamma(G)=\min\{|W|: W \text{ dominates } G\}.
\]
and the {\em independence domination number} $ i\gamma(G)$ is the minimum integer $k$ such that the following holds: for every independent set $I$ of $G$, there exists $W \subset V$ with $|W| \leq k$ which dominates $I$.
Obviously, $i \gamma(G) \leq \gamma(G)$ and $i \gamma(G) = \gamma(G)$ for chordal graphs \cite{ABZ02}.

In \cite{AH00} (see also \cite{M01, M03}), the following relation between the independence domination number of $G$ and homology groups of {\em independence complex}
\[
I(G) := \{W \subset V(G) : W \text{ is independent in } G\}
\]
was proved.
\begin{theorem}\label{eta}
For every graph $G$, $\tilde{H}_i(I(G)) = 0$ for all $i \geq i\gamma(G) - 2$.
\end{theorem}
The (combinatorial) Alexander dual of a simplicial complex $K$ on $V$ is the simplicial complex
$D(K) := \{\sigma \subset V: V \setminus \sigma \notin K\}.$
By observing that $NC(G) = D(I(G))$, Theorem~\ref{eta} combined with the Alexander duality theorem~\cite{BBM97} gives the following.
\begin{cor}\label{eta dual}
For every graph $G$, $\tilde{H}_j(NC(G))=0\text{ for }j \geq n- i\gamma(G)-1$ unless $E(G) = \emptyset$.
\end{cor}

\subsection{Collapsibility of noncover complexes}
For a simplicial complex $K$, a face $\sigma \in K$ is said {\em free} if it is contained in a unique maximal face $\tau \in K$.
An {\em elementary $d$-collapse} of $K$ is an operation on $K$ which removes a free face $\sigma \in K$ with $|\sigma| \leq d$ and all faces containing $\sigma$.
We say $K$ is {\em $d$-collapsible} if we can obtain a void complex from $K$ by a finite sequence of elementary $d$-collapses.
Observe that an elementary $d$-collapse preserves the (non-)vanishing property of homology of dimension at least $d$.
In addition, it is important to notice that any induced subcomplex $L$ of a $d$-collapsibile complex $K$ is also $d$-collapsible.
Therefore every $d$-collapsible complex $K$ is {\em $d$-Leray}, i.e. $\tilde{H}_j(L) = 0$ holds for every induced subcomplex $L$ and $j \geq d$.

In this note, we prove the following which is stronger than Corollary~\ref{eta dual}.
\begin{theorem}\label{main}
Let $G$ be a chordal graph on $n$ vertices with no isolated vertices.
Then $NC(G)$ is $(n - i\gamma(G) - 1)$-collapsible.
\end{theorem}
The proof uses the ``minimal exclusion sequence'' technique.

\subsection{Minimal exclusion sequence}
Let $K$ be a simplicial complex on vertex set $V = \{v_1,\ldots, v_n\}$ where the vertices are ordered by a linear order $v_1 \prec v_2 \prec \cdots \prec v_n$.
Let $\sigma_1,\ldots,\sigma_m$ be the maximal faces of $K$ with a linear order $\sigma_1 \prec_f \sigma_2 \prec_f \cdots \prec_f \sigma_m$.
For each face $\sigma \in X$, we define \[i(\sigma)=\text{min}\{j \in [m]: \sigma \subset \sigma_j\}.\]
We define the {\it minimal exclusion sequence} $\mes(\sigma)=(w_1,w_2,\dots,w_{i-1})$, where $i = i(\sigma)$, as follows:

When $i = 1$, then we assume $\mes(\sigma)$ is an empty sequence.
Otherwise, when $i > 1$, we define $w_j$ for each $j\in [i-1]$ as the minimal element in
\[\begin{cases}\sigma \setminus \sigma_j&\text{ if }(\sigma \setminus \sigma_j) \cap \{w_1,\dots,w_{j-1}\}=\emptyset,
\\(\sigma \setminus \sigma_j) \cap \{w_1,\dots,w_{j-1}\}&\text{otherwise}.\end{cases}\]
Now let $M(\sigma)$ be the set of vertices appeared in the sequence $\mes(\sigma)$ and let $d(K)=\text{max}\{|M(\sigma)| : \sigma \in K\}$.
The following was obtained in \cite{MT09}. (See also \cite{Alan18}.)
\begin{theorem}\label{mes-coll}
Every simplicial complex $K$ is $d(K)$-collapsible.
\end{theorem}

\bigskip
\section{Proof of Theorem~\ref{main}}
Let $G$ be a graph on $V = \{v_1, \ldots, v_n\}$.
Given a linear order $v_1 \prec v_2 \prec \cdots \prec v_n$ on the vertices of $G$, we consider a lexicographic order $\prec_e$ on $E(G)$ as follows:
Let $e = uv$ and $e' = u'v'$ be two distinct edges in $G$ where $u \prec v$ and $u' \prec v'$.
Then $e \prec_e e'$ if and only if either $u \prec u'$ or $u= u'$ and $v \prec v'$.
Let $E(G) = \{e_1,\ldots,e_m\}$ where $e_m \prec_e e_{m-1} \prec_e \cdots \prec_e e_1$.

Now we define a linear order on the set $F(G)$ of all maximal faces on $NC(G)$.
Since any maximal face of $NC(G)$ is the complement of an edge of $G$, we have $|F(G)| = |E(G)| = m$ and $|\sigma| = n-2$ for any $\sigma \in F(G)$.
Let $F(G) = \{\sigma_1,\ldots,\sigma_m\}$ where $\sigma_i = V\setminus e_i$.
We observe that the lexicographic order $\prec_f$ on $F(G)$ gives a linear order $\sigma_1 \prec_f \sigma_2 \prec_f \cdots \prec_f \sigma_m$.

In the following lemma, a {\em star} $S$ is a graph consists of a vertex $v$ and vertices $v_1, \ldots, v_k$ of degree $1$ where each $v_i$ is adjacent to $v$.
The vertex $v$ is called the {\em center} of $S$ and $v_1,\ldots,v_k$ are called the {\em leaves} of $S$.
In other words, a star with $k$ leaves is the complete bipartite graph $K_{1,k}$
\begin{lemma}\label{face-form}
For every $\sigma \in NC(G)$, there exists $\sigma' \in NC(G)$ such that $\mes(\sigma)=\mes(\sigma')$ and the complement $(\sigma')^c$ of $\sigma'$ induces the union of the isolated vertices $b_{k+1},\dots,b_{k+m}$ and a star with the center $a$ and leaves $b_1,\dots,b_k$. 
In addition, $a \prec b_i$ for each $1 \leq i \leq k$.
\end{lemma}
\begin{proof}
Let $\sigma \in NC(G)$ be the maximal face such that $\sigma$ itself does not satisfies the conditions for $\sigma'$.
\begin{enumerate}
\item Suppose the induces subgraph $G[\sigma^c]$ contains two disjoint edges.

Suppose there are two disjoint edges in $G[\sigma^c]$, say $ab, cd \in E(G)$ such that $a \prec b$, $c \prec d$, and $a \prec c$.
Let $\sigma_i = \{a,b\}^c = e_i^c$ and $\sigma_j = \{c,d\}^c = e_j^c$, and define $\tau := \sigma \cup \{a\}$.
It is obvious that $\sigma_j \prec_f \sigma_i$, i.e. $j < i$, and both $\sigma$ and $\tau$ are contained in $\sigma_j$.
We claim that $\mes(\sigma) = \mes(\tau)$.
It is sufficient to show that $\sigma \setminus \sigma_k = \tau \setminus \sigma_k$ for all $1 \leq k \leq j-1$.

For each $1 \leq k \leq j-1$, let $\sigma_k = \{x,y\}^c = e_k^c$ where $x \prec y$.
Since $k < j$, it is obvious that $\sigma_k \prec_f \sigma_j$, and hence $e_j \prec_e e_k$.
Then $c \preceq x$, and it follows that $a \notin e_k$.
Thus 
$$\tau \setminus \sigma_k = \tau \cap e_k = (\sigma \cap e_k)\cup(\{a\} \cap e_k)=\sigma \cap e_k =\sigma \setminus \sigma_k.$$

Now since $|\tau| > |\sigma|$, the complement of $\tau$ induces the union of isolated vertices and a star whose center is the minimal vertex.
Since $\mes(\sigma) = \mes(\tau)$, we take $\sigma' = \tau$.

\item Suppose $\sigma^c$ induces the union of isolated vertices and a star $S$ consists of the center $a$ and at least two leaves such that there exists a leaf $b$ of $S$ with $b \prec a$.

Let $c \neq b$ be another leaf of the star $S$.
Let $\sigma_i = \{a,b\}^c = e_i^c$ and $\sigma_j = \{a,c\}^c = e_j^c$.
We claim that $\mes(\sigma) = \mes(\tau)$ where $\tau = \sigma\cup\{x\}$, $x=\min \{b,c\}$.
Then by the maximality of $\sigma$, we can take $\sigma' = \tau$.
This can be done by modifying the arguments of (i).

If $b \prec c$, then we take $x = b$ and it is enough to show that $b \notin e_k$ for every $k < j$.
For $k < j$, we have $e_j \prec_e e_k$, and it follows from $b \prec a$ and $b \prec c$ that $b \notin e_k$.
Otherwise, if $c \prec b$, then we take $x = c$ and it is enough to show that $c \notin e_k$ for every $k < i$.
For $k < i$, we have $e_i \prec_e e_k$, and it follows from $c \prec b \prec a$ that $c \notin e_k$.

\item Suppose $\sigma^c$ induces the union of isolated vertices and a triange $T$.
Let $a \prec b \prec c$ be three vertices of the triangle $T$.
Since $T$ contains the star with the center $b$ and leaves $a$ ane $c$, the arguments of (ii) shows that $\mes(\sigma) = \mes(\sigma\cup\{a\})$.
\end{enumerate}
\end{proof}

Now let $G$ be a connected chordal graph.
In \cite{Gavril}, it was shown that any chordal graph contains a {\em simplicial vertex}, which is a vertex whose neighbors induced a complete subgraph.
In addition, the following lemma was shown in \cite{Dirac61}.
\begin{lemma}\label{dirac}
If $G$ is a chordal graph which is not a complete graph, then $G$ contains at least two non-adjacent simplicial vertices.
\end{lemma}

Throughout the rest of this document, we use a linear order $v_1 \prec \cdots \prec v_n$ on $V(G)$ which is defined as follows.

Let $v_1$ be an arbitrary simplicial vertex in $G$.
Take $v_n, \ldots, v_{n-|U_1| + 1}$ so that $U_1=\{v_n, \ldots, v_{n-|U_1| + 1}\}$ is the set of all simplicial vertices distinct from $v_1$ in $G$.
If $G$ is not a complete graph, then $U_1 \neq \emptyset$ by Lemma~\ref{dirac}.
Note that if $G$ is a complete graph, then $U_1 = V(G) \setminus \{v_1\}$.
For $i \geq 2$, let $t_i = \sum_{j=1}^{i-1}|U_j|$.
Define 
\[U_i = \{v_{n-t_{i}}, v_{n-t_{i}-1},\ldots,v_{n- t_{i} - |U_{i}| + 1}\}\]
recursively so that $U_i$ is the set of all simplicial vertices in $G[V\setminus(\cup_{j=1}^{i-1}U_j)]$ distinct to $v_1$.
We observe that for every $1 \leq i \leq n$, the induced subgraph $G[\{v_1,\dots,v_i\}]$ is connected.
In the following we summarize the properties of the linear order $v_1 \prec \cdots \prec v_n$ on $V(G)$.

\begin{enumerate}[(i)]
\item\label{111} $v_1$ is a simplicial vertex of $G$.
\item\label{222} $v_i$ is a simplicial vertex of $G[V\setminus\{v_{i+1},\dots,v_n\}]$.
\item\label{333} $G[\{v_1,\dots,v_i\}]$ is connected.
\end{enumerate}

	\begingroup
	\def\thetheorem{\ref{main}}
	\begin{theorem}
	Let $G$ be a chordal graph on $n$ vertices with no isolated vertices.
Then $NC(G)$ is $(n - i\gamma(G) - 1)$-collapsible.
	\end{theorem}
	\addtocounter{theorem}{-1}
	\endgroup
\begin{proof}
Since $i \gamma(G)=\gamma(G)$ for a chordal graph $G$, we will show that $NC(G)$ is $(n-\gamma(G)-1)$-collapsible.
We first show when $G$ is connected.
By Theorem~\ref{mes-coll}, it is sufficient to show $|M(\sigma)| \leq n-\gamma(G)-1$ for any $\sigma \in NC(G)$.
Take a face $\sigma \in NC(G)$.
By Lemma~\ref{face-form}, we may assume that $(\sigma)^c$ induces a star with center $a$ and leaves $b_1,\dots,b_k$ and isolated vertices $b_{k+1},\dots,b_{k+m}$, and $a \prec b_1,\dots,b_k$.

Take $b_{k+m+1},\dots,b_{\gamma(G)+t} \in V(G)$ so that $\{b_1,\dots,b_{\gamma(G)+t}\}$ is a maximal independent set of $G$.
Then $\{b_1,\dots,b_{\gamma(G)+t}\}$ dominates $G$, and hence $t\geq 0$.
If $|M(\sigma) \cap \{b_{k+m+1},\dots,b_{\gamma(G)+t}\}| \leq t$, then
\begin{align*}
|M(\sigma)| &\leq |M(\sigma)\cap (\sigma \setminus \{b_{k+m+1},\dots,b_{\gamma(G)+t}\})|+|M(\sigma)\cap \{b_{k+m+1},\dots,b_{\gamma(G)+t}\}|\\
&\leq |\sigma \setminus \{b_{k+m+1},\dots,b_{\gamma(G)+t}\}|+t=(n-\gamma(G)-t-1)+t=n-\gamma(G)-1,
\end{align*}
as required.

If $|M(\sigma) \cap \{b_{k+m+1},\dots,b_{\gamma(G)+t}\}| \geq t+1$, let 
\[\{c_1,\dots,c_{t+\alpha}\}=M(\sigma) \cap \{b_{k+m+1},\dots,b_{\gamma(G)+t}\}\] for some integer $\alpha \geq 1$.
It is sufficient to find $\alpha$ distinct elements in the set $\sigma \setminus \{b_{k+m+1},\dots,b_{\gamma(G)+t}\}-M(\sigma)$.
Let denote $\mes(\sigma)_j$ by the $j$th entry of $\mes(\sigma)$.
Since $c_i \in M(\sigma)$, we can define $j_i$ by the minimum index such that 
$\mes(\sigma)_{j_i}=c_i$ for each $i \in [t+\alpha]$.
Let $e_{j_i}=\sigma_{j_i}^c=\{c_i,y_i\}$ for some $y_i \in V(G)$.
Then $y_i \in \sigma \setminus \{b_{k+m+1},\dots,b_{\gamma(G)+t}\}$ or $y=a$.
Define two sets
\[
A:=\{c_i: y_i \in \sigma \setminus M(\sigma), i \in [t+\alpha]\}, B:=\{c_1,\dots,c_{t+\alpha}\}\setminus A.
\]
Now we observe the following.
\begin{enumerate}[(a)]
\item\label{aaa} $c_i \prec y_i$ if $y_i \in \sigma$ since $\sigma \setminus \sigma_{j_i}=\sigma \cap e_{j_i}=\{c_i,y_i\}$ and $j_i$ is the minimum index such that $\mes(\sigma)_{j_i}=c_i$.

\item\label{bbb} If $y_i,y_j \in \sigma$ and $i \ne j$, then $y_i \ne y_j$.
Note that $c_i \prec y_i$ and $c_j \prec y_j$ by \eqref{aaa}.
If $y_i=y_j$, then it contradicts to \eqref{222} since $c_i$ and $c_j$ are not adjacent in $G$.

\item\label{ccc} If $y_i \in \sigma \cap M(\sigma)$, then the vertex $a$ is adjacent to $c_i$ in $G$.
Since $y_i \in M(\sigma)$, we can define $l_i$ by the minimum index such that $\mes(\sigma)_{l_i}=y_i$.
Let $e_{l_i}=\sigma_{l_i}^c=\{y_i,y_i'\}$ for some $y_i' \in V(G)$.
Then we obtain $j_i < l_i$ since $\sigma \setminus \sigma_{j_i}=\sigma \cap e_{j_i}=\{c_i,y_i\}$ and $\mes(\sigma)_{j_i}=c_i$.
Thus we have $e_{j_i}=\{c_i,y_i\} \succ_e \{y_i,y_i'\}=e_{l_i}$ and it follows that $c_i \succ y_i'$.
Since $c_i \prec y_i$ by \eqref{aaa}, we have $y_i'\prec y_i$. 
Then by \eqref{222}, $y_i'$ and $c_i$ are adjacent in $G$.
In addition, since $l_i$ is the minimum index such that $\mes(\sigma)_{l_i}=y_i$ and $y_i' \prec y_i$, we know that $y_i' \notin \sigma$.
Since $a$ is the unique vertex in $\sigma^c$ which can be adjacent to $c_i$ in $G$, $y_i=a$.
Therefore $y_i'=a$ is adjacent to $c_i$ in $G$.
\end{enumerate}

Next we define
\[
B_1:=\{c_i \in B : N(c_i) \text{ is dominated by } \{a,b_1,\dots,b_{\gamma(G)+t}\} \setminus B\},
\]
and $B_2:=B \setminus B_1$.
Note that every $c_i \in B$ is adjacent to $a$ in $G$ by \eqref{ccc}.

\begin{claim}\label{claim}
For every $c_i \in B_2$, there exists some \[z_i \in (\sigma\setminus \{b_{k+m+1},\dots,b_{\gamma(G)+t}\})-M(\sigma)\] such that $z_i \succ a$, $z_i \ne z_j$ for two distint $c_i,c_j \in B_2$ and $z_i \notin \{y_i :c_i\in A\}$.
\end{claim}
We first complete the proof of Theorem~\ref{main}, and the proof of Claim~\ref{claim} will appear later.  
By \eqref{bbb} and Claim~\ref{claim},
\[
|(\sigma \setminus \{b_{k+m+1},\dots,b_{\gamma(G)+t}\})\cap M(\sigma)| \leq n-\gamma(G)-t-1-(|A|+|B_2|).
\]
If $|A|+|B_2| \geq \alpha$, then 
\begin{align*}
|M(\sigma)| &\leq n-\gamma(G)-t-1-\alpha+|M(\sigma) \cap \{b_{k+m+1},\dots,b_{\gamma(G)+t}\}|\\
&=n-\gamma(G)-t-1-\alpha+|\{c_1,\dots,c_{t+\alpha}\}|\\
&=n-\gamma(G)-t-1-\alpha+t+\alpha=n-\gamma(G)-1
\end{align*}
as required.

Thus we may assume $|A|+|B_2|\leq \alpha-1$ i.e. $|B_1|=t+\alpha-(|A|+|B_2|)\geq t+1$.
By relabeling and reordering, we can assume that $B_1=\{c_1,c_2,\dots,c_{t'}\}$ for some integer $t'$ such that $t'=|B_2| \geq t+1$.
Since $B_2$ contains at least one element $c_1$, $a$ is adjacent to both $b_1$ and $c_1$, and $b_1$ is not adjacent to $c_1$ in $G$.
Then $a$ is not a simplicial vertex of $G$ and so $a \ne v_1$.
Let $a=v_l$ for some integer $l>1$.

By \eqref{333}, there is a path $P=v_l v_{l_1} \dots v_{l_h}$ in $G$ such that $l>l_1>\cdots>l_h=1$.
Then $v_{l_1} \in \sigma$ since $v_{l_1}$ is adjacent to $a$ in $G$ and $v_{l_1} \prec v_l=a$.
First, suppose that $v_{l_1} \in \sigma \setminus \{b_{k+m+1},\dots,b_{\gamma(G)+t}\}$.
Since $v_{l_1} \prec a$, we have $v_{l_1} \notin M(\sigma)$ from the definiton of $\mes(\sigma)$.
Note that $v_{l_1}$ is distinct to any $y_i$ and $z_j$ since $v_{l_1} \prec a$.
Thus if $|A|+|B_2| \geq \alpha-1$, then $|M(\sigma)|\leq n-\gamma(G)-1$ as required.
Hence we may assume that $|A|+|B_2| \leq \alpha-2$ and then $|B_1| \geq t+2$.
Now we consider the set 
\[X = \{a,b_1,\dots,b_{\gamma(G)+t}\}-B_1.\]
Note that $\{b_1,\dots,b_{\gamma(G)+t}\}$ dominates $G$, $N(c_i)$ is dominated by $X$ for all $c_i\in B_1$, and every $c_i \in B_1$ is adjacent to $a$ in $G$.
It follows that $X$ dominates $G$.
However,
\[
|X|=\gamma(G)+t+1-|B_1| \leq \gamma(G)-1,
\]
which is a contradiction.

Next we suppose $v_{l_1} \notin \sigma \setminus \{b_{k+m+1},\dots,b_{\gamma(G)+t}\}$, i.e. $v_{l_1} \in \{b_{k+m+1},\dots,b_{\gamma(G)+t}\}$.
If $v_{l_1}\ne v_1$, then $v_{l_2} \in \sigma-\{b_{k+m+1},\dots,b_{\gamma(G)+t}\}$ since $v_{l_2}$ is adjacent to $v_{l_1}$ in $G$.
Since $v_{l_2} \prec v_{l_1} \prec v_l=a$, we have $v_{l_2} \notin M(\sigma)$.
Then by a similar argument as the above, one can obtain that $|M(\sigma)| \leq n-\gamma(G)-1$.
Thus we may assume that $v_{l_1}=v_1$.
Since $v_1$ is a simplicial vertex of $G$ by \eqref{111} and $v_1$ is adjacent to $a$ in $G$, every neighbor of $v_1$ distinct to $a$ is adjacent to $a$ in $G$.
Then the set
\[Y = \{a,b_1,\dots,b_{\gamma(G)+t}\}-(B_1\cup \{v_1\})\] 
dominates $G$.
However,
\[
|Y|=\gamma(G)+t+1-(|B_1|+1) \leq \gamma(G)-1,
\]
which is a contradiction.
\end{proof}

\begin{proof}[proof of Claim~\ref{claim}]
Take $c_i \in B_2$.
Since $c_i \notin A$, $y_i=a$ or $y_i \in \sigma \cap M(\sigma)$.
If $y_i=a$, then $c_i$ is adjacent to $y_i=a$ in $G$.
If $y_i \in \sigma \cap M(\sigma)$, then $c_i$ is adjacent to $a$ in $G$ by \eqref{ccc}.
Thus, $c_i$ is adjacent to $a$ in $G$ for both cases.
Since $c_i \notin B_1$, we can take $x_i \in N(c_i)$ such that $x_i$ is not dominated by $\{a,b_1,\dots,b_{\gamma(G)+t}\}\setminus B$.
Then $a$ is not adjacent to $x_i$.
Since $c_i \in M(\sigma)$, $a \prec c_i$ and so $c_i \prec x_i$ by \eqref{222}.
Note that $x_i \in \sigma \setminus \{b_{k+m+1},\dots,b_{\gamma(G)+t}\}$ since $x_i$ is adjacent to $c_i$ in $G$ and $x_i \ne a$.
If $x_i \notin M(\sigma)$, then let $z_i:=x_i$.
Note that $z_i=x_i \succ a$.

Now suppose $x_i \in M(\sigma)$, then we can define $m_i$ by the minimun index such that $\mes(\sigma)_{m_i}=x_i$.
Let $e_{m_i}=\sigma_{m_i}^c=\{x_i,x_i'\}$ for some $x_i' \in V(G)$.
Since $x_i$ is not dominated by $\{a,b_1,\dots,b_{\gamma(G)+t}\}\setminus B$,  $x_i' \notin \{a,b_1,\dots,b_{\gamma(G)+t}\}\setminus B$. 
We want to show that $x_i' \in \sigma \setminus \{b_{k+m+1},\dots,b_{\gamma(G)+t}\}-M(\sigma).$

To show $x_i' \in \sigma \setminus \{b_{k+m+1},\dots,b_{\gamma(G)+t}\}$, it is sufficinet to show $x_i' \notin B$.
Assume that $x_i'=c_j$ for some $c_j \in B$.
By \eqref{ccc} and the definition of $B$, $c_j$ is adjacent to $a$ in $G$ and $c_j \succ a$.
In addition, $x_i \prec x_i'$ since $x_i' \in \sigma$ and the definition of $\mes(\sigma)$.
Now, $x_i'=c_j \succ x_i,a$ but $x_i$ and $a$ are not adjacent in $G$, which contradicts to \eqref{222}.
Thus $x_i' \in \sigma \setminus \{b_{k+m+1},\dots,b_{\gamma(G)+t}\}$.

Now to show $x_i' \notin M(\sigma)$ by contradiction, assume $x_i' \in M(\sigma)$.
Then we can define $k_i$ by the minimum index such that $\mes(\sigma)_{k_i}=x_i'$, and let $e_{k_i}=\sigma_{k_i}^c=\{x_i',x_i''\}$ for some $x_i'' \in V(G)$.
Since $\sigma \setminus \sigma_{m_i}=\{x_i,x_i'\}$ and $\mes(\sigma)_{m_i}=x_i$, $m_i <k_i$.
Thus $e_{m_i}=\{x_i,x_i'\} \succ_e \{x_i',x_i''\}=e_{k_i}$ and so $x_i \succ x_i''$.
Then $x_i'' \prec x_i \prec x_i'$ and so $x_i'' \notin \sigma$.
However, $x_i' \succ x_i,x_i''$ implies that $x_i$ is adjacent to $x_i''$ in $G$ by \eqref{222}, which is a contradiction that $x_i$ is not dominatied by $\{a,b_1,\dots,b_{\gamma(G)+t}\}\setminus B$ since $x'' \in \sigma^c \subset \{a,b_1,\dots,b_{\gamma(G)+t}\}\setminus B$.
Thus, $x_i' \in \sigma \setminus \{b_{k+m+1},\dots,b_{\gamma(G)+t}\}-M(\sigma)$ if $x_i \in M(\sigma)$.
Let $z_i:=x_i'$ if $x_i \in M(\sigma)$.
Note that $z_i=x_i' \succ x_i \succ a$.

Now we will show $z_i \notin \{y_j :c_j \in A\}$ for every $c_i \in B_2$.
Assume $z_i=y_j$ for some $c_j\in A$.
If $z_i=x_i$, then $z_i=x_i \succ c_i$, $x_i=y_j \succ c_j$ but $c_i$ is not adjacent to $c_j$ in $G$, which contradicts to \eqref{222}.
If $z_i=x_i'$, then $z_i=x_i' \succ x_i$, $z_i=y_j \succ c_j$ and so $x_i$ is adjacent to $c_j$ in $G$ by \eqref{222}, which is a contradiction that $x_i$ is not dominated by $\{a,b_1,\dots,b_{\gamma(G)+t}\}\setminus B$.
Thus,  $z_i \notin \{y_i :c_i \in A\}$ for every $c_i \in B_2$.

Now, we only left to show $z_i \ne z_j$ for two distinct $c_i,c_j \in B_2$.
Assume $z_i=z_j$ for some two distinct $c_i,c_j \in B_2$.
If $z_i=x_i$ and $z_j=x_j$, then $z_i=z_j \succ c_i,c_j$ and $c_i$ and $c_j$ are not adjacent in $G$, which contradicts to \eqref{222}.
Suppose $z_i=x_i$ and $z_j=x_j'$, then $x_j'=x_i \succ x_j,c_i$ and so $x_j$ and $c_i$ are adjacent in $G$.
If $x_j \succ c_i$, then $x_j \succ c_i,c_j$ which is a contradiction that $c_i$ and $c_j$ are not adjacent in $G$ because of \eqref{222}.
If $x_j \prec c_i$, then $c_i \succ a,x_j$ which is a contradiction that $a$ and $x_j$ are not adjacent in $G$ because of \eqref{222}.
Finally, assume $z_i=x_i'$ and $z_j=x_j'$.
Then $x_i'=x_j' \succ x_i,x_j$ and so $x_i$ and $x_j$ are adjacent in $G$ by \eqref{222}.
Without loss of generality, we may assume $x_i \succ x_j$.
Then $x_i \succ x_j,c_i$ and so $x_j$ and $c_i$ are adjacent in $G$ by \eqref{222}.
If $x_j \succ c_i$, then $x_j \succ c_i,c_j$ which is a contradiction that $c_i$ and $c_j$ are not adjacent in $G$ because of \eqref{222}.
If $x_j \prec c_i$, then $c_i \succ a,x_j$ which is a contradiction that $a$ and $x_j$ are not adjacent in $G$ because of \eqref{222}.
\end{proof}


\bibliographystyle{alpha}

%

\end{document}